\documentclass[a4paper,12pt]{article}
\usepackage[latin1]{inputenc}
\usepackage[T1]{fontenc}
\usepackage{amsmath}
\usepackage{amsfonts}
\usepackage{amssymb}
\usepackage{amsthm}
\usepackage{color}
\usepackage{amsthm}
\usepackage{graphicx}
\usepackage{hyperref}
\newtheorem{thm}{Theorem}[section]
\newtheorem{defn}[thm]{Definition}
\newtheorem{rmk}[thm]{Remark}
\newtheorem{prop}[thm]{Proposition}
\newtheorem{lem}[thm]{Lemma}

\usepackage{indentfirst}

\title{On the cover time of $\lambda$-biased walk on supercritical Galton-Watson trees}
\author{Tianyi Bai\footnote{Laboratoire Analyse G\'{e}om\'{e}trie et Applications,  UMR 7539, CNRS, Universit\'{e} Paris 13 - Sorbonne Paris Cit\'{e}, Universit\'{e} Paris 8, 99, Avenue Jean-Baptiste Cl\'{e}ment, 93430 Villetaneuse. Email: bai@math.univ-paris13.fr}}
\date{}
\begin{document}
\maketitle
\begin{abstract}
In this paper, we study the time required for a $\lambda$-biased ($\lambda>1$) walk to visit all the vertices of a supercritical Galton-Watson tree up to generation $n$. Inspired by the extremal landscape approach in \cite{2019} for the simple random walk on binary trees, we establish the scaling limit of the cover time in the biased setting.
\end{abstract}
\bibliographystyle{plain}

\section{Introduction}
\subsection{The model}
A planar tree $\mathbf T$ is a subset of $\cup_{n\ge 0}\mathbb N^n_+$ such that
\begin{itemize}
    \item The root $\varnothing$ is in $\mathbf T$, where by convention, $\mathbb N_+^0=\{\varnothing\}$.
    \item For every vertex $x=(x_1,\cdots,x_n)\in \mathbf T$, its parent $$\overleftarrow x=(x_1,\cdots,x_{n-1})\in\mathbf T.$$
    \item There exists an integer $\nu_x(\mathbf T)\ge 0$ representing the number of children of $x$, i.e., for every $j\in\mathbb N_+,$ 
    $$(x_1,\cdots,x_n,j)\in \mathbf T \text{ if and only if } 1\le j \le \nu_x(\mathbf T).$$
\end{itemize}

For $x\in\mathbf T$, denote by $|x|=n$ its height. For two vertices $x,y\in\mathbf T$, write $x\preceq y$ if $x$ is on the simple path from $\varnothing$ to $y$. Denote by $z=x\wedge y$ the common ancestor, i.e. the vertex with maximum height $|z|$ such that $z\preceq x, y$. Denote the tree $\mathbf T$ chopped at height $n$ by $\mathbf T_n=\{x\in\mathbf T,|x|\le n\}$ (in $\mathbf {T}_n$, $\nu_x=0$ if $|x|=n$), and
the population in generation $n$ by
$Z_n=\sum_{x\in\mathbf T}\mathbf 1_{|x|=n}.$
Denote the subtree of $\mathbf T$ rooted at $x$ by
$\mathbf T^{x}=\{y\in\mathbf T,x\preceq y\},$
and the population in the $n$-th generation of $\mathbf T^{x}$ by
$Z_n^{x}=\sum_{y\in \mathbf T^{x}}\mathbf 1_{|y|=n}.$ For convenience of further usage (local time related calculations), an artificial root $\overleftarrow\varnothing$ with height $-1$ is added to be the parent of $\varnothing$. 

For any fixed probability distribution $\mu=(\mu_n)_{n\in\mathbb N}$, the Galton-Watson tree with offspring distribution $\mu$ is a measure $\mathbf P_{\text{GW}}$ on the set of planar trees, such that all the vertices have children distributed identically and independently as $\mu$, i.e. 
$$\nu_x\overset{iid}{\sim}\mu,\forall x\ne\overleftarrow\varnothing.$$
The expectation under this probability measure is denoted by $\mathbf E_{\text{GW}}$.

Let $m=\sum_{i=0}^\infty i\mu_i$ denote the average number of children for a vertex (except $\overleftarrow\varnothing$), then there is a standard result on Galton-Watson trees that, if $m>1$ (so-called supercritical), then the tree extends to infinity with a positive probability. In other words, denote by $\mathcal S$ the event that a Galton-Watson tree survives, $\{\mathbf T:Z_n(\mathbf T)>0,\forall n\ge 0\}$, then $\mathbf P_{\text{GW}}(\mathcal S)>0$ for $m>1$. 
To study the asymptotic behavior of the cover time, we study the first $n$ generations of a supercritical Galton-Watson tree under the conditional probability measure $\mathbf P_{\text{GW}}(\cdot|\cal S)$. 

By the Kesten-Stigum theorem (cf. \cite{kesten1966limit}), $(\frac{Z_n}{m^n})_{n\ge 0}$ is a martingale, and that it converges to a nontrivial limit $\mathbf P_{\text{GW}}(\cdot|\cal S)$-almost surely,
\begin{equation}\label{Wpositive}
W:=\lim_{n\rightarrow\infty}\frac{Z_n}{m^n}\in(0,\infty),
\end{equation}
when $\mathbf E_{\text{GW}} (Z_1\log Z_1)=\sum_{k\ge 1}k\log k\mu_k<\infty$.

Given a surviving tree $\mathbf T$ and fix $n>0$, let $(X_n(t))_{t\ge 0}$ be a continuous time Markov jump process  on $\mathbf T_n$ starting at $X_n(0)=\overleftarrow\varnothing$, with transition rates $p(\overleftarrow\varnothing,\varnothing)=1$,
$$p(x,\overleftarrow x)=\frac{\lambda}{\lambda+\nu_x},\,p(x,x^{(i)})=\frac{1}{\lambda+\nu_x},\forall x\in\mathbf T_n\backslash\{\overleftarrow\varnothing\},1\le i\le\nu_x.$$
The probability measure for this random walk is denoted by $\mathbf P_{\rm w}(\cdot|\mathbf T_n)$, and its corresponding expectation by $\mathbf E_{\rm w}(\cdot|\mathbf T_n)$. In fact, it is more natural to call the discrete walk $\lambda$-biased random walk, and the results for both settings are the same, see Remark \ref{thmremark} (1).

We shall work under the following hypotheses denote by ({\bf H}),
$$\lambda>1,\,m>1,\,\sum_{k\ge 0}k^2\mu_k<\infty.$$
\paragraph{Conventions:}
\begin{enumerate}
\item The probability of a generic law independent of the constructions above is denoted by $\mathbb P$, with its expectation $\mathbb E$. 
\item When an integer index is needed in the presence of a real value, we mean its integral part, for instance, we write $Z_{\log n}$ for $Z_{[\log n]}$. 
\item We write $f\lesssim g$, if there exists a constant $C>0$ such that $f\le Cg.$ This constant may depend on parameters $m$ and $\lambda$.
\item All the $O(\cdot)$, $o(\cdot)$ notations are under the limit $n\rightarrow \infty$.
\item For convenience of the readers, notations used throughout the paper are gathered here:
$$\sigma_n=\sqrt{\frac{\lambda^{n+1}-1}{\lambda-1}},\,t_n^\mu=\sigma_n^2(\log Z_n+\mu),\,s_n=\sum_{i=0}^n\frac{Z_i}{\lambda^i}.$$
\end{enumerate}
\subsection{Main results}
The goal in this paper is to estimate the cover time defined by
$$T^{cov}_{n}(\mathbf T)=\inf\left\{t:\{X_n(s),0\le s\le t\}=\mathbf T_n\right\}.$$
The main result is
\begin{thm}\label{main}
Under {\rm({\bf H})}, for $\mathbf P_{\rm{GW}}(\cdot|\cal S)$-almost surely any tree $\mathbf T$, with $x\in\mathbb R$ and $n\rightarrow\infty$, 
when $\lambda>m$, 
$$\mathbf P_{\rm w}\left(\frac{(\lambda-1)T_n^{cov}}{2\lambda^{n+1}\sum_{i=0}^{\infty}\frac{Z_i}{\lambda^i}}-n\log m-\log W\le x\middle|\mathbf T_n\right)\rightarrow e^{-e^{-x}};$$
when $\lambda=m$, 
$$\mathbf P_{\rm w}\left(\frac{(m-1)T_n^{cov}}{2m^{n+1}\sum_{i=0}^n Z_i}-n\log m-\log W\le x\middle|\mathbf T_n\right)\rightarrow e^{-e^{-x}};$$
when $1<\lambda<m$, 
$$\mathbf P_{\rm w}\left(\frac{(\frac{m}{\lambda}-1)(\lambda-1)T_n^{cov}}{2Wm^{n+1}}-n\log m-\log W\le x\middle|\mathbf T_n\right)\rightarrow e^{-e^{-x}}.$$
\end{thm}
\begin{rmk}\label{thmremark}
\begin{enumerate}
\item The the same result holds for the discrete time random walk with the same transition probabilities, since $(X_n(t))_{0\le t\le t_0}$ takes $t_0+O(\sqrt {t_0})$ steps, and the error is negligible for the cover time.
\item In fact, the conditions in Lemma \ref{typical} are the only requirements for $\mathbf T$. Therefore, the result applies to other (random) trees satisfying these conditions, not necessarily the Galton-Watson trees.
\item The $\lambda$-biased case agrees with the case of the simple random walk (cf. eg. \cite{drznew}) in first order in the number of excursions (round trips from $\overleftarrow\varnothing$) performed, for details see Remark \ref{gffcontrast}.
\item If one define the random walk on $\mathbf T$ instead of $\mathbf T_n$, the relation of $\lambda$ and $m$ discussed in Theorem \ref{main} correspond to whether the walk is transient, positive recurrent, or null recurrent.
\end{enumerate}
\end{rmk}
\subsection{Related works}

The cover time of a finite graph (by the simple random walk), $T^{cov}(G)$, is a fundamental object for a finite graph $G=(V,E)$ (Section 2, Lov\'{a}sz \cite{lovasz1993random}). For a simple graph $G$ with $|V|=n$, a tight bound for its cover time was given in Feige \cite{feige1995tight}, \cite{feige1995tight2} 
$$(1-o(1))n\log n\le T^{cov}(G)\le 4n^3/27.$$

Bounds using hitting time were given in Matthews \cite{matthews1988covering},
$$\max_{S\subseteq G}\min_{u,v\in S}H(u,v)(\log (\#E)-1)\le T^{cov}(G)\le\max_{u,v\in G} H(u,v)(1+\log n),$$
where $H(u,v)$ is the expected time that the walk takes from $u$ to $v$.

Up to the first order approximation, a general bound with Discrete Gaussian Free Field (DGFF) was given in Ding, Lee and Peres \cite{JDcover}, then improved in Zhai \cite{zhai2014exponential}, 
$$\mathbb P\left(\left|T^{cov}(G)-\#EM^2\right|\ge \#E(\sqrt{s R}M+s R)\right)\le Ce^{-cs},$$
where $M=\mathbb E(\max_{x\in V}\eta_x)$, $R=\max_{x,y\in V}R_{\text{eff}}(x,y)$, $(\eta_x)_{x\in V}$ is a DGFF on $G$ (centered Gaussian variables such that $\mathrm{Cov}(\eta_x,\eta_y)=R_{\text{eff}}(x,y)$), and $R_{\text{eff}}$ is the effective resistance (cf. \cite{aldous1995reversible}, consider each edge as a wire of electrical resistance $1$, and take effective resistance in the physics sense, following Ohm's law).

Sharper results can be obtained if one restricts to trees. The first order estimate for the cover time of an $m$-ary trees was obtained in Aldous \cite{Aldous},
$$T^{cov}_n=(2+o(1))n^2 \frac{m^{n+1}}{m-1}\log m.$$
It is showed in Andreoletti and Debs \cite{Andr} that, the first $R_n=(\gamma+o(1))\log n$ generations are covered in $n$ steps, by a recurrent Markovian random walk on the Galton-Watson tree, where $\gamma$ is an explicit constant.

The case of simple random walk on binary trees received extensive studies recently, originally as a counterexample showing that at second order, the cover time is no longer determined by the DGFF (cf. \cite{JDsharp}). Second order asymptotics with error $O(\log\log^8 n)$ were given in Ding and Zeitouni \cite{JDsharp}, then refined to $O(1)$ in Belius, Rosen and Zeitouni \cite{sharpplus}, and a scaling limit was established independently by Cortines, Louidor and Saglietti \cite{2019} and Dembo, Rosen and Zeitouni \cite{drznew},
$$\mathbb P\left(\frac{T^{cov}_n}{2^{n+1}n}-n\log 2+\log n\le s\right)=\mathbb E\exp\left(-CZe^{-s}\right),$$
for some explicit constant $C$ and distribution $Z$.

In the studies of the cover time, the continuous counterpart for trees is the two-dimensional torus. The first order estimate of its cover time was determined by Dembo, Peres, Rosen and Zeitouni \cite{dembo2004cover}, then the result was improved in Ding \cite{ding2011cover}, Belius and Kistler \cite{belius2017subleading}, Abe \cite{abe2017second}, and most recently studied by Belius, Rosen and Zeitouni \cite{belius2017tightness} to the extent that
$$\lim_{K\rightarrow\infty}\limsup_{\epsilon\rightarrow 0}\mathbb P\left(\left|\sqrt{T^{cov}_{\epsilon}(M)}-\sqrt{\frac{2A_M}{\pi}}\left(\log\epsilon^{-1}-\frac{1}{4}\log\log\epsilon^{-1}\right)\right|>K\right)=0,$$
where $M$ is a 2-dimensional manifold with some regularity conditions, $A_M$ is the area of $M$, and $T^{cov}_{\epsilon}(M)$ is the time for the walk to intersect every ball of radius $\epsilon$ on $M$. 

\subsection{Proof outline}

The cover time is determined by the excursion time and local times defined as follows.
\begin{defn} \label{basicdefn} Let $\mathbf T$ be an infinite tree, and let $(X_n(t))_{t\ge 0}$ be a random walk on $\mathbf T_n$.
\begin{enumerate}
\item The excursion time is denoted by
$$t_n^{cov}=\int_0^{T_n^{cov}}\mathbf{1}_{\{X_n(s)=\overleftarrow\varnothing\}}ds.$$
\item To establish the relation between $t^{cov}_n$ and $T^{cov}_n$ , let
$$\tau_n(t)=\inf\left\{r>0:\int_0^r\mathbf{1}_{\{X_n(s)=\overleftarrow \varnothing\}}ds\ge t\right\}.$$
\item The (normalized) local time at $x\in\mathbf T_n\backslash\{\overleftarrow\varnothing\}$ is
$$L^x_{n}(t)=\frac{1}{\pi_n(x)}\int_0^{\tau_n(t)}\mathbf{1}_{\{X_n(s)=x\}}ds,$$
where the normalization factor
$$\pi_n(x)=\frac{\lambda+\nu_x}{\lambda^{|x|+1}},\,x\in\mathbf T_n\backslash\{\overleftarrow\varnothing\},$$
is the stationary distribution of the $\lambda$-biased walk scaled at $\pi_n(\overleftarrow\varnothing)=1$.
\end{enumerate}
\end{defn}

Intuitively, the walk $(X_n(t))$ can be seen as independent samples of random walks that starts and ends at $\overleftarrow \varnothing$. Each of these trials is called an excursion. Then $t^{cov}_n$ is the number of excursions performed to cover the tree $\mathbf T_n$, $\tau_n(t)$ is the actual time spent in $t$ excursions, and $L^x_n(t)$ encodes the status of $x$ in $t$ excursions. From these definitions, we have
\begin{equation}\label{TtRelation}
\begin{aligned}
&\tau_n(t^{cov}_n)\le T_n^{cov}\le \lim_{\epsilon\rightarrow 0+}\tau_n(t_n^{cov}+\epsilon)\\
&\tau_n(t)=\sum_{x\in \mathbf T_n}\pi_n(x)L^x_n(t).
\end{aligned}
\end{equation}

In fact, it is not hard to determine $L^x_t$ (Lemma \ref{LTcharacter}) and $\tau_n$ (Lemma \ref{variancecal}), and the main part of the proof is to estimate $t^{cov}_n$. The key observation is that, when the tree is almost covered (at the first order estimate of the cover time), the non-visited vertices can be seen as independent. (This is inspired by the extremal landscape structure in \cite{2019}, see Remark \ref{gffcontrast} (1) for details.) The scaling limit of the cover time is then established by characterizing the process afterwards.

The paper is organized as follows. In Section \ref{basicnotion}, we give the  regularity conditions on trees and determine the distribution of local times $L^x_n$. In Section \ref{neariid}, we establish the scaling limit for $t^{cov}_n$. And in Section \ref{relation}, we estimate $T^{cov}_n$ by studying $\tau_n$, and finish the proof of Theorem \ref{main}.

\section{Preliminaries}\label{basicnotion}
\subsection{The trees}

\begin{lem}\label{typical}
Let $c,\epsilon>0$. Under {\rm({\bf H})}, for $\mathbf P_{\rm{GW}}(\cdot|\cal S)$-almost surely any tree $\mathbf T$, when $n$ is large enough,
\begin{align}
\label{213}\sum_{i=0}^{n}Z_i&\lesssim Z_n,\\
\label{211}\left|Z_n-m^nW\right|&\le m^{n/2}\log n,\\
\label{212}\sum_{|x|=n-c\log n}(Z^x_n)^2&\le Z_n^{1+\epsilon}.
\end{align}
\end{lem}
\begin{proof}
By (\ref{Wpositive}), $\mathbf P_{\text{GW}}(\cdot|\cal S)$-almost surely, $W\in (0,\infty)$. Therefore, there exists a constant $C>1$ (depending on $\mathbf T$) such that for all $n\ge 0$, 
\begin{equation*}
\frac{1}{C}Wm^n<Z_n<CWm^n.
\end{equation*}
(There are only finitely many $n$ violating this relation with $C=2$, take the maximum constant among these $n$.) One can then deduce (\ref{213}).

Moreover, by Theorem 2, \cite{heyde}: $\mathbf P_{\text{GW}}(\cdot|\mathcal S)$-almost surely, 
$$\limsup_{n\rightarrow\infty}\left|\frac{Z_n-m^nW}{\sqrt{Z_n\log n}}\right|=\sqrt{\frac{2\mathrm{Var}\mu}{m^2-m}},$$
where we recall that $\mu$ is the offspring distribution of $\mathbf P_{\text{GW}}$.

Therefore one can take $n$ large enough such that
\begin{equation*}
\begin{aligned}
|Z_n-m^nW|\lesssim \sqrt{Z_n\log n}\le \sqrt{CWm^n\log n},
\end{aligned}
\end{equation*}
and (\ref{211}) follows.

By standard calculations, 
$$\mathbf E_{\text{GW}}Z_n=m^{n}, \mathrm{Var}_{\text{GW}} (Z_n)=m^{n-1}\frac{m^n-1}{m-1}\mathrm{Var} \mu\lesssim m^{2n},$$
therefore, for $n$ large enough,
\begin{equation*}
\begin{aligned}
\mathbf P_{\text{GW}}\left(\sum_{|x|=n-c\log n}(Z^x_n)^2> m^{(1+\epsilon)n}\middle| \cal S\right)&\le\frac{\mathbf E_{\text{GW}}(\sum_{|x|=n-c\log n}(Z^x_n)^2|\cal S)}{m^{(1+\epsilon)n}}\\
&\lesssim \frac{\mathbf E_{\text{GW}}(Z_{n-c\log n})m^{2c\log n}}{m^{(1+\epsilon)n}}\\
&\lesssim m^{-\frac{\epsilon n}{2}},
\end{aligned}
\end{equation*}
and (\ref{212}) follows from the union bound,
$$\mathbf P_{\text{GW}}\left(\exists n>N,\sum_{|x|=n-c\log n}(Z^x_n)^2> m^{(1+\epsilon)n}\middle|\cal S\right)\lesssim \sum_{n>N}m^{-\frac{\epsilon n}{2}}\overset{N\rightarrow\infty}{\longrightarrow}0.$$
We remark that $\sum_{|x|=n-r}(Z_n^x)^2$ is monotone increasing in $r$, therefore, the event (\ref{212}) is decreasing in $r$, and once it is valid for $c\log n$, it is valid for all $0\le r\le c\log n$ simultaneously.
\end{proof}
\subsection{The local times}
\begin{defn}\label{reffsigmapg}Let $\mathbf T$ be a tree.
\begin{enumerate}
\item The effective resistance (for the $\lambda$-biased walk) between $x,y\in\mathbf T$ is
\begin{equation*}
R_{\rm{eff}}(x,y)=\sum_{\substack{(z,\overleftarrow z)\\ \text{ on the simple path}\\ \text{from }x\text{ to }y}}\lambda^{|z|}.
\end{equation*}
\item The effective resistance between $\overleftarrow\varnothing$ and any vertex at generation $n$ is abbreviated as
$$\sigma_n^2={\frac{\lambda^{n+1}-1}{\lambda-1}}.$$ 
\item For $a,b>0$, let $\mathtt{PG}(a,b)$ be the distribution of $\sum_{i=1}^P E_i$, where $P$ and $E_i$ are independent random variables such that $P\sim\mathtt{Poisson}(a)$ has Poisson distribution of expected value $a$, and $E_i\sim\mathtt{Exp}(b)$ has exponential distribution of expected value $\frac{1}{b}$.
\end{enumerate}
\end{defn}
The following lemma characterizes local times.
\begin{lem}\label{LTcharacter}
On any infinite tree $\mathbf T$, let $x,y\in\mathbf T_n$ such that $y\prec x$ (i.e. $y$ is a strict ancestor of $x$), let $s,t>0$. Then under $\mathbf P_{\rm w}(\cdot|\mathbf T_n)$,
\begin{align*}
L_n^x(t)&\sim\mathtt{PG}\left(\frac t {\sigma_{|x|}^2},\frac 1 {\sigma_{|x|}^2}\right),\\
\left(L_n^x(t)\middle| L_n^{y}(t)=s\right)&\sim\mathtt{PG}\left(\frac{s}{\sigma_{|x|}^2-\sigma_{|y|}^2},\frac{1}{\sigma_{|x|}^2-\sigma_{|y|}^2}\right).
\end{align*}
\end{lem}
\begin{proof}
By the memoryless property of the exponential distribution (if $X\sim\mathtt{Exp}(1)$, then $(X-c|X>c)\overset{d}{=}X$), $L^x_n(t)$ is only affected by local times on the ray from $\overleftarrow\varnothing$ to $x$, independent of movements on other branches or offspring of $x$.

By knowledge of reversible Markov chains (cf. eg. (3.24), p.69, \cite{aldous1995reversible}, it can be easily checked that $R_{\text{eff}}$ defined in Definition \ref{reffsigmapg} do correspond to that of a reversible Markov chain, since there is only one simple path between two vertices),
\begin{equation*}
\frac{1}{R_{\text{eff}}(x,y)}=\pi(x)\mathbf P_x(\tau_y<\tau_x^+|\mathbf T_n),
\end{equation*}
where $\mathbf P_x$ is the $\lambda$-biased walk starting at $x$, $\tau_y$ is the first hitting times of $y$, i.e. $\tau_y=\inf\{t>0,X_n(t)=y\}$, and $\tau_x^+$ is the first returning time to $x$, $\tau_x^+=\inf\{t>0,X_n(t)=x,X_n(s)\not\equiv x,0\le s\le t\}$. In particular, $$\mathbf P_x(\tau_{\overleftarrow\varnothing}<\tau_x^+|\mathbf T_n)=\frac{1}{\sigma_{|x|}^2\pi_n(x)}.$$

Up to excursion time $t$, there are $\mathtt{Poisson}(t)$ departures from $\overleftarrow\varnothing$, and by the equation above, each trip hits $x$ independently with probability $\frac{1}{\sigma_{|x|}^2}$, thus there are $\mathtt{Poisson}\left(\frac{t}{\sigma_{|x|}^2}\right)$ arrivals on $x$. Upon arrival at $x$, the walk returns to $\overleftarrow\varnothing$ in exponential time, with rate $\mathbf P_x(\tau_{\overleftarrow\varnothing}<\tau_x^+|\mathbf T_n)=\frac{1}{\sigma_{|x|}^2\pi_n(x)}$. Therefore, the total time spent at $x$ has distribution $\mathtt{PG}\left(\frac{t}{\sigma_{|x|}^2},\frac{1}{\sigma_{|x|}^2\pi_n(x)}\right)$. Recall that local times are normalized by $\frac{1}{\pi_n}$, and the result follows.

Conditioned at $L^y_n(t)$, the proof is similar. (Only to notice that the local times at both $x,y$ are normalized.)
\end{proof}

\begin{lem}\label{PGextreme}Let $a,b>0$, let $X\sim\mathtt{PG}(a,b).$
\begin{enumerate}
\item The $\mathtt{PG}$ distribution has basic properties
$$bX\sim\mathtt{PG}(a,1),\,\mathbb E(X)=\frac{a}{b},\,\mathrm{Var}(X)=\frac{2a}{b^2},\,\mathbb P(X=0)=e^{-a}.$$
\item If $a>b$, then 
$$\mathbb P(X\le 1)\le e^{2\sqrt{ab}-a-b}.$$
\end{enumerate}
\end{lem}
\begin{proof}
(1) is clear by definition. For (2), by Chernoff bounds, for any $\theta>0$,
\begin{equation*}
\begin{aligned}
\mathbb P(X\le 1)&=\mathbb P\left(e^{-\theta X}\ge e^{-\theta}\right)\\
&\le e^{\theta}\sum_{k=0}^{\infty} e^{-a}\frac{a^k}{k!}\left(1+\frac{\theta}{b}\right)^{-k}=e^{\theta-\frac{\theta}{b+\theta}a},
\end{aligned}
\end{equation*}
and the result follows by choosing $\theta=\sqrt {ab}-b$.
\end{proof}

\subsection{Ray-Knight Theorem}
\begin{defn}\label{gff}
A Discrete Gaussian Free Field (DGFF) on a tree $\mathbf T$ is a family of random variables $(\eta_x)_{x\in\mathbf T}$ such that 
$\eta_{\overleftarrow\varnothing}=0$, $(\eta_x)_{x\ne\overleftarrow\varnothing}$ are centered Gaussian variables with (both the effective resistance and the DGFF can be defined up to any scale, the factor $\frac{1}{2}$ is taken in accordance to \cite{2019})
$$\mathbb{E}(\eta_x-\eta_y)^2=\frac{1}{2}R_{\rm{eff}}(x,y).$$
\end{defn}
\begin{rmk}\label{dgffeasyconstruct}
Since we have explicit relative resistances, if we attach an independent Gaussian variable $N_y\sim\mathcal N\left(0,\frac{\lambda^{|y|}}{2}\right)$ at each vertex $y\in\mathbf T\backslash\{\overleftarrow\varnothing\}$, and let $\eta_x=\sum_{y\preceq x}N_y$, then $(\eta_x)_{x\in \mathbf T}$ is a DGFF on $\mathbf T$. 
\end{rmk}
\begin{thm}{(Second Ray-Knight theorem, \cite{RK})}\label{rk}
	For any infinite tree $\mathbf T$, let $(\eta_x)_{x\in\mathbf T}$ be a DGFF on $\mathbf T$ independent of $\mathbf P_{\rm w}(\cdot|\mathbf T_n)$. For any $t>0$,
$$\left\{L_n^x(t)+\eta^2_x:x\in \mathbf T_n\right\}\overset{d}{=}\left\{(\eta_x+\sqrt{t})^2:x\in \mathbf T_n\right\}.$$
\end{thm}
In fact, Remark \ref{dgffeasyconstruct} indicates a direct proof of this theorem by induction.

\begin{lem}\label{extremeGFF} 
Let $(\eta_x)_{x\in\mathbf T}$ be a DGFF on a tree $\mathbf T$, let $n\ge 0$ and $\mu\in\mathbb R$ such that
$$Z_n>0,\,\log Z_n+\mu>0,$$
we have
$$\mathbb P\left(\max_{|x|=n}\eta_x>\sigma_n \sqrt{\log Z_n+\mu}\right)\le \frac{e^{-\mu}}{2\sqrt{\pi(\log Z_n+\mu)}}.$$
\end{lem}
\begin{proof}
We first recall the Gaussian tail estimate for $x>0$ and $X\sim\mathcal{N}(0,1)$,
\begin{equation*}
\begin{aligned}
\mathbb P(X>x)
=\frac{1}{\sqrt {2\pi}}\int_0^\infty e^{-(y+x)^2/2}dy
\le\frac{e^{-x^2/2}}{\sqrt {2\pi}}\int_0^\infty e^{-xy}dy=\frac{e^{-x^2/2}}{x\sqrt {2\pi}}.
\end{aligned}
\end{equation*}
Then by the union bound, when $\log Z_n+\mu>0,$
\begin{equation*}
\begin{aligned}
\mathbb P\left(\max_{|x|=n}\eta_x>\sigma_n \sqrt{\log Z_n+\mu}\right)
\le Z_{n}\frac{e^{-(\log Z_n+\mu)}}{2\sqrt{\pi(\log Z_n+\mu)}}=\frac{e^{-\mu}}{2\sqrt{\pi(\log Z_n+\mu)}}.
\end{aligned}
\end{equation*}
\end{proof}

\section{Excursion time}\label{neariid}
As showed in the proof outline, the excursion time $t^{cov}_n$ is compared to the quantity 
\begin{equation}\label{aboutcover}
t_n^\mu=\sigma^2_n(\log Z_n+\mu),
\end{equation}
where we recall that
$$\sigma_n^2=\frac{\lambda^{n+1}-1}{\lambda-1}.$$
To start with, we show that the first $n-c\log n$ layers have negligible influences in terms of local times.
\begin{lem}\label{RegulationAncestor}
Let $c>\frac{3}{\log\lambda}$, $\mu\in\mathbb R$, and $\mathbf T$ be an infinite tree satisfying the conditions in Lemma \ref{typical}.
For $n\rightarrow\infty$, with probability $1-o(1)$ under $\mathbf P_{\rm w}(\cdot|\mathbf T_n)$,
\begin{equation*}
\max_{|x|=n-c\log n} \left|\frac{L^x_n(t_n^\mu)-t_n^\mu}{\sigma_n^2}\right|\lesssim \frac{1}{\sqrt n}.
\end{equation*}
\end{lem}
\begin{proof}
Let $(\eta_x)$ be a DGFF on $\mathbf T$ independent of $\mathbf P_{\rm w}(\cdot|\mathbf T_n)$. Denote the joint law of the DGFF and $\mathbf P_{\rm w}(\cdot|\mathbf T_n)$ by $\mathbb P(\cdot|\mathbf T_n)$ in this proof.

It is guaranteed by Lemma \ref{typical} that $\log Z_n+\mu>0$ for $n$ large enough. Then one can apply Lemma \ref{extremeGFF},
\begin{equation*}
\mathbb P\left(\max_{|x|=n-c\log n}|\eta_x|\le\sigma_{n-c\log n}\sqrt{\log Z_{n-c\log n}+\mu}\middle|\mathbf T_n\right)=1-o(1).
\end{equation*}
Moreover, by Theorem \ref{rk},
\begin{equation*}
\mathbb P\left(\max_{|x|=n-c\log n}\left|\sqrt{L^x(t_n^\mu)+\eta_x^2}-\sqrt {t_n^\mu}\right|\le\sigma_{n-c\log n}\sqrt{\log Z_{n-c\log n}+\mu}\middle|\mathbf T_n\right)=1-o(1).
\end{equation*}

Thus with probability $1-o(1)$, for $n$ large enough,
\begin{equation*}
\begin{aligned}
&\quad\;\max_{|x|=n-c\log n}|L^x(t_n^\mu)-t_n^\mu|\\
&=\max_{|x|=n-c\log n}\left|\left(\sqrt{L^x(t_n^\mu)+\eta_x^2}-\sqrt{t_n^\mu}\right)\left(\sqrt{L^x(t_n^\mu)+\eta_x^2}+\sqrt{t_n^\mu}\right)-\eta_x^2\right|\\
&\le\max_{|x|=n-c\log n}\left|\sqrt{L^x(t_n^\mu)+\eta_x^2}-\sqrt {t_n^\mu}\right|^2\\
&\qquad\qquad\quad+2\sqrt{t^\mu_n}\max_{|x|=n-c\log n}\left|\sqrt{L^x(t_n^\mu)+\eta_x^2}-\sqrt {t_n^\mu}\right|+\max_{|x|=n-c\log n}|\eta_x|\\
&\lesssim \sigma_n^2\lambda^{-c\log n/2}\left(\log Z_{n}+\mu\right)\lesssim \sigma_n^2n^{-1/2},
\end{aligned}
\end{equation*}
where the last line follows from Lemma \ref{typical} and the condition $c>\frac{3}{\log\lambda}.$
\end{proof}

Now we present the key observation that, non-visited vertices up to time $t^\mu_n$ have distinct ancestors at generation $n-c\log n$.
\begin{lem}\label{StructureOfF}
Let $c>\frac{3}{\log\lambda}$, $\mu\in\mathbb R$, and $\mathbf T$ be an infinite tree satisfying the conditions in Lemma \ref{typical}.
There exists $\epsilon>0$ such that 
$$\max_{\substack{|y\wedge z|\ge n-c\log n+1,\\|y|=|z|=n,y\ne z}}\mathbf P_{\rm w}\left(L_n^y(t_n^\mu)=L^z_n(t_n^\mu)=0\,\middle|\,(L_{n}^x(t_n^\mu))_{|x|=n-c\log n},\mathbf T_n\right)=o\left(Z_n^{-1-\epsilon}\right),$$
uniformly in any choices of local times at generation $n-c\log n$ satisfying the conditions in Lemma \ref{RegulationAncestor}.
\end{lem}
\begin{proof}
For simplicity, the conditional probability $\mathbf P_{\rm w}\left(\cdot|(L_{n}^x(t_n^\mu))_{|x|=n-c\log n},\mathbf T_n\right)$ is abbreviated as $\mathbf P_{\rm w}(\cdot|R_n)$, and we write $L^x$ for the local time $L^x_n(t^\mu_n)$.

Let $\delta>0$. Fix two vertices $y,z$ at generation $n$ of $\mathbf T_n$, such that they have common ancestor $w=y\wedge z$ after generation $n-c\log n$. Let $|w|=n-s\ge n-c\log n+1$, and denote by $x$ the ancestor of $w$ with $|x|=n-\log n$. Then
\begin{equation*}
\begin{aligned}
&\quad\;\mathbf P_{\rm w}\left(L^y=L^z=0\,\middle|\,R_n\right)\\
&\le \mathbf P_{\rm w}\left(L^y=L^z=0,L^w\ge 2\delta t_n^\mu\,\middle|\,R_n\right)+\mathbf P_{\rm w}\left(L^w< 2\delta t_n^\mu\,\middle|\,R_n\right)\\
&=\mathbf P_{\rm w}\left(L^z=0\,\middle|\,L^y=0,L^w\ge2\delta t_n^\mu,R_n\right)\mathbf P_{\rm w}\left(L^y=0,L^w\ge 2\delta t^\mu_n\,\middle|\,R_n\right)+\mathbf P_{\rm w}\left(L^w< 2\delta t_n^\mu\,\middle|\,R_n\right)\\
&=\mathbf P_{\rm w}\left(L^z=0\,\middle|\,L^w\ge2\delta t_n^\mu,R_n\right)\mathbf P_{\rm w}\left(L^y=0,L^w\ge 2\delta t^\mu_n\,\middle|\,R_n\right)+\mathbf P_{\rm w}\left(L^w< 2\delta t_n^\mu\,\middle|\,R_n\right)\\
&\le \mathbf P_{\rm w}\left(L^z=0\,\middle|\,L^w=2\delta t_n^\mu,R_n\right)\mathbf P_{\rm w}\left(L^y=0\,\middle|\,R_n\right)+\mathbf P_{\rm w}\left(L^w< 2\delta t_n^\mu\,\middle|\,R_n\right).
\end{aligned}
\end{equation*}

By Lemma \ref{RegulationAncestor}, for $n$ large enough we have
$$L^x>(1-\delta)t_n^\mu,$$
moreover, by Lemma \ref{LTcharacter}, 
$$(L^w|L^x)\overset{d}{=}\mathtt{PG}\left(\frac{L^x}{\sigma^2_{n-s}-\sigma^2_{n-c\log n}},\frac{1}{\sigma^2_{n-s}-\sigma^2_{n-c\log n}}\right).$$
where denominator above is bounded by
$$\frac{\lambda-1}{\lambda}\sigma^2_{n-s}\le\sigma^2_{n-s}-\sigma^2_{n-s-1}\le\sigma^2_{n-s}-\sigma^2_{n-c\log n}\le \sigma^2_{n-s}.$$
Therefore, for local times of the generation $n-c\log n$ satisfying Lemma \ref{RegulationAncestor} and for $n$ large enough, 
\begin{equation*}
\begin{aligned}
&\quad\;\mathbf P_{\rm w}\left(L^w<2\delta t_n^\mu\,\middle|\,R_n\right)\\
&=\mathbf P_{\rm w}\left(\mathtt{PG}\left(\frac{L^x}{\sigma^2_{n-s}-\sigma^2_{n-c\log n}},\frac{1}{\sigma^2_{n-s}-\sigma^2_{n-c\log n}}\right)< 2\delta t_n^\mu\,\middle|\,R_n\right)\\
&\le \mathbb P\left(\mathtt{PG}\left(\frac{(1-\delta)t_n^\mu}{\sigma^2_{n-s}},\frac{\lambda/(\lambda-1)}{\sigma^2_{n-s}}\right)\le 2\delta t_n^\mu\right)\\&\le e^{-\lambda\left(\sqrt{1-\delta}-\sqrt{2\delta\lambda/(\lambda-1)}\right)^2( \log Z_n+\mu)},
\end{aligned}
\end{equation*} 
where the notation $\mathbb P(\mathtt{PG}(a,b)\le x)$ stands for the probability that a random variable  distribution as $\mathtt{PG}(a,b)$ is smaller than $x$, and the last line is due to Lemma \ref{PGextreme}.

Similarly, study $(L^z|L^w)$ and $(L^y|L^x)$ by Lemma \ref{LTcharacter}, for $n$ large enough,
\begin{equation*}
\begin{aligned}
&\quad\;\mathbf P_{\rm w}\left(L^z=0\,\middle|\,L^w=2\delta t_n^\mu,R_n\right)\mathbf P_{\rm w}\left(L^y=0\,\middle|\,R_n\right)\\
&=e^{-\frac{2\delta t_n^\mu}{\sigma_n^2-\sigma^2_{n-s}}-\frac{L^x}{\sigma^2_n-\sigma^2_{n-c\log n}}}\le e^{-\frac{2\delta t_n^\mu}{\sigma_n^2}-\frac{(1-\delta)t_n^\mu}{\sigma^2_n}}= e^{-(1+\delta)(\log Z_n+\mu)}.
\end{aligned}
\end{equation*}

To show that the two probabilities above are bounded by $o\left(Z_n^{-1-\epsilon}\right)$, it suffices to choose $\delta,\epsilon>0$ such that
\begin{equation}\label{lowerbounddelta}
\begin{aligned}
\lambda \left(\sqrt{1-\delta}-\sqrt{2\delta\lambda/(\lambda-1)}\right)^2> 1+\epsilon,\\
1+\delta> 1+\epsilon,
\end{aligned}
\end{equation}
which is always possible when $\lambda>1$.
\end{proof}

The estimate above is enough to control the cover time of the $n$-th generation, however a Galton-Watson tree may have leaves in younger generations, in other words, the walk visits all vertices in the $n$-th generation does not guarantee that it covers $\mathbf T_n$. We treat leaves before the generation $n-1$ separately.
\begin{lem}\label{badgwcase}
Let $\mu\in\mathbb R$, and let $\mathbf T$ be an infinite tree satisfying the conditions in Lemma \ref{typical}, then
$$\mathbf P_{\rm w}\left(\exists x\in\mathbf T_n,|x|\le n-1, L^x_n(t_n^\mu)=0\middle| \mathbf T_n\right)=o(1).$$
\end{lem}
\begin{proof}
By Lemma \ref{LTcharacter} and Lemma \ref{PGextreme} (1), each $x\in\mathbf T_n$ with $|x|\le n-1$ is not visited at excursion time $t_n^\mu$ with probability 
$$\mathbf P_{\rm w}(L_n^x(t_n^\mu)=0|\mathbf T_n)={e^{-{\sigma_{|x|}^{-2}}{t_n^\mu}}}\le e^{-\lambda(\log Z_n+\mu)},$$ thus by the union bound,
\begin{equation*}
\begin{aligned}
\mathbf P_{\rm w}\left(\exists x\in\mathbf T_n,|x|\le n-1, L^x_n(t_n^\mu)=0\middle| \mathbf T_n\right)&\le e^{-\lambda(\log Z_n+\mu)}\sum_{|x|\le n-1} 1\\
&=e^{-\lambda(\log Z_n+\mu)}\sum_{i=0}^{n-1} Z_i,
\end{aligned}
\end{equation*}
and the conclusion follows from (\ref{213}).
\end{proof}

Returning to the $n$-th generation, by Lemma \ref{StructureOfF}, the non-visited vertices are almost independent at excursion time $t_n^\mu$. Intuitively, the time to cover them is a binomial random variable of $Z_{n-c\log n}$ trials, converging to a Poisson distribution. We conclude upon this intuition:
\begin{prop}  \label{mainexcursion}
Let $\mu\in\mathbb R$, and let $\mathbf T$ be an infinite tree satisfying the conditions in Lemma \ref{typical}, recall $t_n^\mu$ defined in (\ref{aboutcover}), we have
\begin{gather}
\label{341}\mathcal \#\left\{|x|=n,L_n^x(t_n^\mu)=0\right\}\overset{d}{\rightarrow}\mathtt{Poisson}(e^{-\mu})\text{ under }\mathbf P_{\rm w}(\cdot|\mathbf T_n), \\
\label{342}\mathbf P_{\rm w}\left(\frac{t_n^{cov}}{\sigma_n^2}-n\log m-\log W\le \mu\middle|\mathbf T_n\right)\rightarrow e^{-e^{-\mu}}.
\end{gather}
\end{prop}
\begin{proof}
Since both the conclusions allow a $o(1)$ error, we may assume the conditions in Lemma \ref{RegulationAncestor} at generation $n-c\log n$ for an arbitrarily fixed $c>\frac{3}{\log\lambda}$. For simplicity, recall the abbreviation $\mathbf P_{\rm w}(\cdot|R_n)$ in the proof of Lemma \ref{StructureOfF}, then it suffices replace $\mathbf P_{\rm w}(\cdot|\mathbf T_n)$ by $\mathbf P_{\rm w}(\cdot|R_n)$.
Moreover, denote by
$$F_n^\mu=\left\{|x|=n,L_n^x(t_n^\mu)=0\right\},E_n^\mu=\left\{|x|=n-c\log n+1,\#(\mathbf T^{x}_n\cap F_n^\mu)=1\right\},$$
the goal is to estimate $F_n^\mu$, which is achieved by comparison to $E_n^\mu$.

By Lemma \ref{LTcharacter},  under $\mathbf P_w(\cdot|R_n)$, local times are distributed as $$(L^y_n(t^\mu_n)|R_n)\sim\mathtt{PG}\left(\frac{L_n^{x}(t_n^\mu)}{\sigma^2_n-\sigma^2_{n-c\log n}},\frac{1}{\sigma^2_n-\sigma^2_{n-c\log n}}\right)$$ 
for each pair of vertices 
$$x\prec y,|x|=n-c\log n,|y|=n.$$ 
Therefore, by Lemma \ref{RegulationAncestor},
\begin{equation*}
\begin{aligned}
\mathbf E_{\rm w}\left(\#F_n^\mu\,\middle|\,R_n\right)
&=\sum_{|x|=n-c\log n}e^{-\frac{L^x_n(t_n^\mu)}{\sigma_n^2-\sigma_{n-c\log n}^2}}Z_n^x\\
&=\sum_{|x|=n-c\log n}e^{-\log Z_n-\mu+o(1)}Z_n^x=e^{-\mu+o(1)}.
\end{aligned}
\end{equation*}

Furthermore, by definition one has
$$0\le \#F_n^\mu-\#E_n^\mu\le 2\sum_{|x|=n-c\log n+1}\sum_{y,z\in \mathbf T^x_n}\mathbf 1_{\{y,z\in F_n^\mu\}},$$
where $\mathbf E_{\rm w}(\mathbf 1_{\{y,z\in F_n^\mu\}}|R_n)=o(Z_n^{-(1+\epsilon)})$
by Lemma \ref{StructureOfF}.

Therefore, by (\ref{212}),
\begin{equation*}
\begin{aligned}
\mathbf E_{\rm w}\left(\#E_n^\mu\,\middle|\,R_n\right)&=\mathbf E_{\rm w}(\#F_n^\mu|R_n)+\mathbf E_{\rm w}(\#E_n^\mu-\#F_n^\mu|R_n)\\
&=e^{-\mu+o(1)}+o\left(Z_n^{-(1+\epsilon)}\sum_{|x|=n-c\log n+1}(Z^x_n)^2\right)\\&=e^{-\mu+o(1)}+o(1),
\end{aligned}
\end{equation*}
where $\epsilon$ is the parameter in Lemma \ref{StructureOfF}.

Moreover, conditioned on local times of layer $n-c\log n$, the subtrees $(\mathbf T^x)_{|x|=n-c\log n+1}$ are independent, thus for any $\theta>0$, the Laplace transform of $\#E_n^\mu$ is given by
\begin{equation*}
\begin{aligned}
&\;\quad\mathbf E_{\rm w}\left(e^{-\theta\#E_n^\mu}\,\middle|\,R_n\right)\\
&=\prod_{|x|=n-c\log n+1}\mathbf E_{\rm w}\left(e^{-\theta\cdot\mathbf 1_{x\in E_n^\mu}}\,\middle|\,R_n\right)\\
&=\prod_{|x|=n-c\log n+1}\left(1-\left(1-e^{-\theta}\right)\mathbf P_{\rm w}\left(x\in E_n^\mu\,\middle|\,R_n\right)\right)\\
&=\prod_{|x|=n-c\log n+1}e^{-\left(1-e^{-\theta}+o(1)\right)\mathbf P_{\rm w}\left({x\in E_n^\mu}\middle|R_n\right)}\\
&=e^{-\left(1-e^{-\theta}+o(1)\right)\mathbf E_{\rm w}\left(\#E_n^\mu\,\middle|\,R_n\right)}\rightarrow e^{-(1-e^{-\theta})e^{-\mu}},
\end{aligned}
\end{equation*}
in other words
$$\#E_n^\mu\overset{d}{\rightarrow}\mathtt{Poisson}(e^{-\mu}).$$
Finally, by Lemma \ref{StructureOfF} and the union bound again, we have
$$\mathbf P_{\rm w}\left(\#E_n^\mu=\#F_n^\mu\middle|R_n\right)=1-o(1),$$
thus $\#F_n^\mu$ has the same distributional limit as $\#E_n^\mu$, completing the proof of (\ref{341}).

As for (\ref{342}), $\{t^{cov}_n\le t_n^\mu\}$ differs from $\{\#F_n^\mu=0\}$ by whether the leaves of $\mathbf T_n$ in early generations are covered, which is controlled in Lemma \ref{badgwcase}. Therefore, by (\ref{341}),
$$\mathbf P_{\rm w}\left(t^{cov}_n\le t_n^\mu|R_n\right)=\mathbf P_{\rm w}(\#F_n^\mu=0|R_n)+o(1)\rightarrow e^{-e^{-\mu}},$$
then (\ref{342}) follows from the asymptotic of $t_n^\mu$, $$\frac{t_n^\mu}{\sigma_n^2}-n\log m-\log W\rightarrow \mu.$$
\end{proof}

\begin{rmk}\label{gffcontrast}
\begin{enumerate}
\item For the complete binary tree, we have $m=2$. Take $\lambda\rightarrow 1$, then $\mathbf P_{\rm w}(\cdot|\mathbf T_n)$ converges to a simple random walk, our result gives (non-rigorously)
$$t_n^{cov}\approx n^2\log 2 +O(n),$$
whereas the cover time on the binary tree of a simple random walk is (cf. eg. \cite{drznew})
$$t^{cov}_n=n^2\log 2-n\log n+O(n).$$
Lack of the second order term $n\log n$ is due to a difference in extremal landscapes: recall the notations in the proof of Proposition \ref{mainexcursion}, in the case of a simple random walk \cite{2019}, the set $F_n^\mu$ is approximately identically distributed clusters indexed by $E_n^\mu$, whereas for the $\lambda$-biased walk, these clusters are single points instead.
\item Following exactly the same structure of the proof, one can study the maximum of a DGFF $(\eta_x)$ (recall Definition \ref{gff}),
$$\mathbb P\left(\max_{|x|=n}\eta_x\le\sigma_n\sqrt{\log Z_n-\frac{1}{2}\log\log Z_n+\mu}\right)\rightarrow \exp\left(-\frac{e^{-\mu}}{2\sqrt\pi}\right).$$
\item Comparing the cover time to the maximum of the corresponding DGFF, as suggested in \cite{JDcover}, \cite{JDsharp}, 
one has
\begin{equation*}
\begin{aligned}
t^{cov}_n&=\frac{\lambda^{n+1}}{\lambda-1}(n\log m+O(1)),\\
\max_{|x|=n}\eta_x^2&=\frac{\lambda^{n+1}}{\lambda-1}\left(n\log m-\frac{1}{2}\log n+O(1)\right).
\end{aligned}
\end{equation*}
This difference in second order is due to different tails of Gaussian and local time distributions. (Compare Lemma \ref{PGextreme} (2) and Lemma \ref{extremeGFF}.)
\end{enumerate}
\end{rmk}
\section{From excursion time to real time}\label{relation}
\begin{lem}\label{variancecal}
For any infinite tree $\mathbf T$ satisfying the conditions in Lemma \ref{typical},
let $s_n=\sum_{i=0}^n\frac{Z_i}{\lambda^i}$, $t>0$, then
\begin{align}
\label{411}\mathbf E_{\rm w}\left(\tau_n(t)\middle|\mathbf T_n\right)&=\mathbf E_{\rm w}\left(\sum_{x\in\mathbf T_n}\pi_n(x)L^x_n(t)\middle|\mathbf T_n\right)=2ts_n,\\
\label{412}\mathrm{Var}_{\rm w}(\tau_n(t)|\mathbf T_n)&=o\left(\frac{t\lambda^{n}s_n^2}{n}\right).
\end{align}
\end{lem}
\begin{proof}
The expected value (\ref{411}) is clear using the fundamental estimates $\mathbf E_{\rm w} (L^x(t)|\mathbf T_n)=t$ and $\sum_{|x|=k}\nu_x=Z_{k+1}$. As for (\ref{412}), for any $x,y\in\mathbf T_n$, conditioned at $L^{x\wedge y}_n(t)$, local times $L^x_n(t)$ and $L^y_n(t)$ are independent with the same expected value $L^{x\wedge y}_n(t)$, thus by Lemma \ref{PGextreme} (1),
$$\mbox{Cov}_{\rm w}(L^x_n(t),L_n^y(t)|\mathbf T_n)=\mathrm{Var}_{\rm w}(L_n^{x\wedge y}(t)|\mathbf T_n)=2t\sigma^2_{|x\wedge y|}\le 2t\frac{\lambda}{\lambda-1}\lambda^{|x\wedge y|},$$
then by (\ref{TtRelation}), 
\begin{equation*}
\begin{aligned}
&\;\quad \mathrm{Var}_{\rm w}(\tau_n(t)|\mathbf T_n)\\
&=\mathrm{Var}_{\rm w} \left(\sum_{x\in\mathbf T_n}\pi(x)L^x_n(t)\middle|\mathbf T_n\right)\\
&=\sum_{x,y\in \mathbf T_n}\pi(x)\pi(y)\mathrm{Cov}_{\rm w}(L^x_n(t),L^y_n(t)|\mathbf T_n)\\
&\le 2t\frac{\lambda}{\lambda-1}\sum_{x,y\in\mathbf T_n}\lambda^{|x\wedge y|-|x|-|y|}\left(1+\frac{\nu_x}{\lambda}+\frac{\nu_y}{\lambda}+\frac{\nu_x\nu_y}{\lambda^2}\right),
\end{aligned}
\end{equation*}
where $\nu_x$ is the number of children for $x\in\mathbf T_n$.

Moreover, 
$$\lambda^{|x\wedge y|-|x|-|y|}\frac{\nu_x}{\lambda}=\sum_{\overleftarrow z=x}\lambda^{|z\wedge y|-|z|-|y|},$$
therefore the variance above is further bounded by
\begin{equation*}
\mathrm{Var}_{\rm w}(\tau_n(t)|\mathbf T_n)\le 8t\frac{\lambda}{\lambda-1}\sum_{x,y\in\mathbf T_n}\lambda^{|x\wedge y|-|x|-|y|}.
\end{equation*}

Now it suffices to prove that 
$$\sum_{x,y\in\mathbf T_n}\lambda^{|x\wedge y|-|x|-|y|}=o\left(\frac{\lambda^{n}s_n^2}{n}\right).$$
Indeed, fix any $c>0$,
\begin{equation*}
\begin{aligned}
&\quad\;\sum_{x,y\in\mathbf T_n}\lambda^{|x\wedge y|-|x|-|y|}\\
&\le \sum_{|x\wedge y|< n-c\log n}\lambda^{n-c\log n-|x|-|y|}+\sum_{|x\wedge y|\ge n-c\log n}\lambda^{n-|x|-|y|}\\
&\le \lambda^{n-c\log n}\left(\sum_{x\in\mathbf T_n}\lambda^{-|x|}\right)^2+\sum_{|x\wedge y|\ge n-c\log n}\lambda^{n-2(n-c\log n)}\\
&\le \frac{\lambda^{n}s_n^2}{n^2}+\lambda^{2c\log n-n}\sum_{|x|=n-c\log n}(Z^x_n)^2,\\
\end{aligned}
\end{equation*}
where by (\ref{212}) and definition of $s_n$,
\begin{equation*}
\begin{aligned}
&\quad\;\lambda^{2c\log n-n}\sum_{|x|=n-c\log n}(Z^x_n)^2\\
&\le\lambda^{2c\log n-n}Z_n^{1+\epsilon}\\
&\le\lambda^{2c\log n-n}(s_n\lambda^n)^{1+\epsilon}=o\left(\frac{\lambda^ns_n^2}{n}\right).
\end{aligned}
\end{equation*}
\end{proof}

\begin{prop}\label{maincore}
Recall $s_n$ from Lemma \ref{variancecal}, and recall that $\sigma_n^2=\frac{\lambda^{n+1}-1}{\lambda-1}$. Under {\rm({\bf H})},
for any $\mu\in \mathbb R$ and $\mathbf P_{\rm{GW}}(\cdot|\cal S)$-almost surely any tree $\mathbf T$,
$$\mathbf P_{\rm w}\left(\frac{T_n^{cov}}{2s_n\sigma_n^2}-n\log m-\log W\le \mu\middle| \mathbf T_n\right)\rightarrow e^{-e^{-\mu}}.$$
\end{prop}
\begin{proof}
By Lemma \ref{typical},
$$\log Z_n=n\log m+\log W+o(1),$$
therefore it suffices to show that (rigorously speaking, one should prove the following convergence for $\mu\pm\epsilon$, then take $\epsilon\rightarrow 0$ to deduce the proposition), 
$$\mathbf P_{\rm w}\left(\frac{T_n^{cov}}{2s_n\sigma_n^2}- \log Z_n\le\mu\middle|\mathbf T_n\right)\rightarrow e^{-e^{-\mu}}.$$
In other words (recall $t_n^\mu$ defined in (\ref{aboutcover})), it suffices to prove that
$$\mathbf P_{\rm w}\left(\frac{T_n^{cov}}{2s_n\sigma_n^2}\le \frac{t_n^\mu}{\sigma_n^2}\middle|\mathbf T_n\right)\rightarrow e^{-e^{-\mu}}.$$

By (\ref{TtRelation}), for any $\alpha>0$,
\begin{equation*}
\begin{aligned}
\mathbf P_{\rm w}(T_n^{cov}\le 2s_nt_n^\mu|\mathbf T_n)&\le\mathbf P_{\rm w}(\tau_n(t_n^{cov})\le 2s_nt_n^\mu|\mathbf T_n)\\
&\le \mathbf P_{\rm w}(t_n^{cov}\le t_n^{\mu+\alpha}|\mathbf T_n)\\
&\qquad+\mathbf P_{\rm w}(t^{cov}_n>t_n^{\mu+\alpha},|\tau_n(t_n^{\mu+\alpha})-2t_n^{\mu+\alpha}s_n|>2s_n(t^{\mu+\alpha}_n-t^\mu_n)|\mathbf T_n).
\end{aligned}
\end{equation*}

For the first term, by Proposition \ref{mainexcursion}, 
$$\mathbf P_{\rm w}(t_n^{cov}\le t_n^{\mu+\alpha}|\mathbf T_n)\le (1+o(1))e^{-e^{-\mu-\alpha}}.$$
For the second term, by Chebyshev's inequality and Lemma \ref{variancecal}
\begin{equation*}
\begin{aligned}
&\quad\;\mathbf P_{\rm w}(t^{cov}_n>t_n^{\mu+\alpha},|\tau_n(t_n^{\mu+\alpha})-2t_n^{\mu+\alpha}s_n|>2s_n(t^{\mu+\alpha}_n-t^\mu_n)|\mathbf T_n)\\
&\le\mathbf P_{\rm w}(|\tau_n(t_n^{\mu+\alpha})-2t_n^{\mu+\alpha}s_n|>2s_n(t^{\mu+\alpha}_n-t^\mu_n)|\mathbf T_n)\\
&=o\left(\frac{\frac{1}{n}\lambda^nt^{\mu+\alpha}_n}{(t^{\mu+\alpha}_n-t^\mu_n)^2}\right)=\alpha^{-2}o\left(1\right).
\end{aligned}
\end{equation*}
Let $\alpha\rightarrow 0+$, we have
$$\limsup_{n\rightarrow\infty}\mathbf P_{\rm w}(T_n^{cov}\le 2s_nt_n^\mu|\mathbf T_n)\le e^{-e^{-\mu}}.$$
Similarly, for any $\alpha>0$, and any $\beta(\alpha)>0$ small enough,
\begin{equation*}
\begin{aligned}
&\;\quad\mathbf P_{\rm w}(T_n^{cov}\le 2s_nt_n^\mu|\mathbf T_n)\\
&\ge \mathbf P_{\rm w}\left(\tau_n(t_n^{cov}+\beta)\le 2s_nt_n^\mu,t^{cov}_n\le t^{\mu-\alpha}_n|\mathbf T_n\right)\\
&=\mathbf P_{\rm w}\left(t^{cov}_n\le t^{\mu-\alpha}_n|\mathbf T_n\right)-\mathbf P_{\rm w}\left(\tau_n(t_n^{cov}+\beta)\ge 2s_nt_n^\mu,t^{cov}_n\le t^{\mu-\alpha}_n|\mathbf T_n\right)\\&\rightarrow e^{-e^{-\mu+\alpha}}.
\end{aligned}
\end{equation*}
\end{proof}
\begin{proof}[Proof of Theorem \ref{main}]
By Proposition \ref{maincore}, it suffices to estimate $s_n$.

For $1<\lambda<m$, one can use (\ref{211}) to show that
$$s_n-\sum_{i=0}^n\frac{m^i}{\lambda^i}W=O\left(\frac{m^{n/2}}{\lambda^{n/2}}+\sum_{i=n/2}^n\frac{m^{i/2}\log n}{\lambda^i}\right)=o\left(\frac{m^n}{n\lambda^n}\right).$$
Therefore, $s_n$ can be replaced by $(\frac{m}{\lambda}-1)^{-1}\frac{m^{n+1}}{\lambda^{n+1}}W$, and the conclusion follows.

For $\lambda>m$, similarly, the difference between $\sum_{i=0}^\infty\frac{Z_i}{\lambda^i}$ and $s_n$ is negligible,
$$\sum_{i=0}^{\infty}\frac{Z_i}{\lambda^i}-s_n=\sum_{i=n+1}^{\infty}\frac{Z_i}{\lambda^i}=O \left(\frac{m^n}{\lambda^n}\right).$$

For $\lambda=m$, $s_n=\sum_{i=0}^n Z_i$ follows from its definition in Lemma \ref{variancecal}.
\end{proof}
{\bf Acknowledgement}. The author would like to thank Yueyun Hu for pointing out this topic, and providing valuable advice. I would also like to thank Yijun Wan and  the reviewers for various suggestions on this paper.
\bibliography{aa}

\begin{thebibliography}{10}

\bibitem{abe2017second}
Yoshihiro Abe.
\newblock Second order term of cover time for planar simple random walk.
\newblock 2017.
\newblock arXiv:1709.08151.

\bibitem{Aldous}
David~J. Aldous.
\newblock Random walk covering of some special trees.
\newblock {\em J. Math. Anal. Appl.}, 157(1):271--283, 1991.

\bibitem{aldous1995reversible}
David~J. Aldous and James Fill.
\newblock Reversible markov chains and random walks on graphs, 1995.

\bibitem{Andr}
Pierre Andreoletti and Pierre Debs.
\newblock The number of generations entirely visited for recurrent random walks
  in a random environment.
\newblock {\em J. Theor. Probab.}, 27(2):518--538, 2014.

\bibitem{belius2017subleading}
David Belius and Nicola Kistler.
\newblock The subleading order of two dimensional cover times.
\newblock {\em Probab. Theory Rel.}, 167(1-2):461--552, 2017.

\bibitem{belius2017tightness}
David Belius, Jay Rosen, and Ofer Zeitouni.
\newblock Tightness for the cover time of compact two dimensional manifolds.
\newblock 2017.
\newblock arXiv:1711.02845.

\bibitem{sharpplus}
David Belius, Jay Rosen, and Ofer Zeitouni.
\newblock Barrier estimates for a critical {G}alton-{W}atson process and the
  cover time of the binary tree.
\newblock {\em Ann. Inst. Henri Poincar\'{e} Probab. Stat.}, 55(1):127--154,
  2019.

\bibitem{2019}
Aser Cortines, Oren Louidor, and Santiago Saglietti.
\newblock A scaling limit for the cover time of the binary tree.
\newblock 2018.
\newblock arXiv:1812.10101.

\bibitem{dembo2004cover}
Amir Dembo, Yuval Peres, Jay Rosen, and Ofer Zeitouni.
\newblock Cover times for {B}rownian motion and random walks in two dimensions.
\newblock {\em Ann. of Math.}, 160(2):433--464, 2004.

\bibitem{drznew}
Amir Dembo, Jay Rosen, and Ofer Zeitouni.
\newblock Limit law for the cover time of a random walk on a binary tree.
\newblock 2019.
\newblock arXiv:1906.07276.

\bibitem{ding2011cover}
Jian Ding.
\newblock On cover times for 2{D} lattices.
\newblock {\em Electron. J. Probab.}, 17:no. 45, 18, 2012.

\bibitem{JDcover}
Jian Ding, James~R. Lee, and Yuval Peres.
\newblock Cover times, blanket times, and majorizing measures.
\newblock {\em Ann. of Math. (2)}, 175(3):1409--1471, 2012.

\bibitem{JDsharp}
Jian Ding and Ofer Zeitouni.
\newblock A sharp estimate for cover times on binary trees.
\newblock {\em Stoch. Proc. Appl.}, 122(5):2117--2133, 2012.

\bibitem{RK}
Nathalie Eisenbaum, Haya Kaspi, Michael~B. Marcus, Jay Rosen, and Zhan Shi.
\newblock A {R}ay-{K}night theorem for symmetric {M}arkov processes.
\newblock {\em Ann. Probab.}, 28(4):1781--1796, 2000.

\bibitem{feige1995tight}
Uriel Feige.
\newblock A tight lower bound on the cover time for random walks on graphs.
\newblock {\em Random Struct. Algor.}, 6(4):433--438, 1995.

\bibitem{feige1995tight2}
Uriel Feige.
\newblock A tight upper bound on the cover time for random walks on graphs.
\newblock {\em Random Struct. Algor.}, 6(1):51--54, 1995.

\bibitem{heyde}
Christopher~C. Heyde and Julian~R. Leslie.
\newblock Improved classical limit analogues for galton-watson processes with
  or without immigration.
\newblock {\em B. Aust. Math. Soc.}, 5(2):145--155, 1971.

\bibitem{kesten1966limit}
Harry Kesten and Bernt~P Stigum.
\newblock A limit theorem for multidimensional galton-watson processes.
\newblock {\em Ann. Math. Statist.}, 37(5):1211--1223, 1966.

\bibitem{lovasz1993random}
L\'{a}szl\'{o} Lov\'{a}sz.
\newblock Random walks on graphs: a survey.
\newblock In {\em Combinatorics, {P}aul {E}rd\H{o}s is eighty}, pages 353--397.
  J\'{a}nos Bolyai Math. Soc., Budapest, 1993.

\bibitem{matthews1988covering}
Peter Matthews.
\newblock Covering problems for {M}arkov chains.
\newblock {\em Ann. Probab.}, 16(3):1215--1228, 1988.

\bibitem{zhai2014exponential}
Alex Zhai.
\newblock Exponential concentration of cover times.
\newblock {\em Electron. J. Probab.}, 23:no. 32, 22, 2018.

\end{thebibliography}
\end{document}